\documentclass[a4paper,11pt]{article}
\usepackage{amsmath}
\usepackage{amsfonts}
\usepackage{supertabular}
\usepackage[latin9]{inputenc}
\usepackage[english]{varioref}
\usepackage{dcolumn}
\usepackage[height=22cm , width = 16cm , top = 4cm , left = 3cm, a4paper]{geometry}
\usepackage[a4paper]{geometry}
\usepackage[final]{graphicx}
\usepackage{epsfig}
\usepackage{pstricks}
\usepackage{psfrag}
\usepackage{rotating}
\usepackage{supertabular}
\usepackage{booktabs}
\usepackage{delarray}
\usepackage{rotating}
\usepackage{subfigure}
\usepackage{nextpage}
\usepackage{layout}
\usepackage{amsthm}
\usepackage{dsfont}
\usepackage{hyperref}
\usepackage[hypcap]{caption}
\usepackage{color}
\numberwithin{equation}{section}

\theoremstyle{plain}
\newtheorem{thm}{Theorem}[section]

\newtheorem{lem}[thm]{Lemma}

\newtheorem{defn}[thm]{Definition}

\newtheorem{hypp}[thm]{Hypotheses}
\newcommand{\enter}{\bigskip}

\date{January 10, 2017}
\begin{document}
 \author{{Prasanta Kumar Barik and
 Ankik Kumar Giri}\vspace{.2cm}\footnote{Corresponding author. Tel +91-1332-284818 (O);  Fax: +91-1332-273560  \newline{\it{${}$ \hspace{.3cm} Email address: }}ankikgiri.fma@iitr.ac.in/ankik.math@gmail.com}\\
\footnotesize \small{ \textit{Department of Mathematics, Indian Institute of Technology Roorkee,  Roorkee-247667, Uttarakhand,}}\\ \small{\textit{India}}
  }

\title{A note on mass-conserving solutions to the coagulation-fragmentation equation by using non-conservative approximation }

\maketitle

%%%%%%%%%%%%%%%%%%%%%%%%%%%% %%%%%%%%%%%%%%%%%
\hrule \vskip 11pt

\begin{quote}
{\small {\em\bf Abstract.}In general, the non-conservative approximation of coagulation-fragmentation equations (CFEs) may lead to the occurrence of gelation phenomenon. In this article, it is shown that the non-conservative approximation of CFEs can also provide the existence of mass conserving solutions to CFEs for large classes of unbounded coagulation and fragmentation kernels.\enter
}
\end{quote}

\vskip 11pt \hrule

%%%%%%%%%%%%%%%%%%%%%%%%%%%%%%%%%%%%%%%%%%%%%%%%%%%%%%%%%%%%%%%%%%%
%%%%%%%%%%%%%%%%%%%%%%%%%%%%%%%%%%%%%%%%%%%%%%%%%%%%%%%%%%%%%%%%%%%
\section{Introduction}
An area of substantial concern in engineering and science is the phenomenon of particulate coagulation
and fragmentation in chemical and biochemical process systems such as crystallization, fluidization and activated sludge flocculation. The basic reactions taken into consideration are coalescing of two particles to form a larger particle and the breakage of particles into two daughter fragments.
The coagulation-fragmentation equations (CFEs) are widely used to describe the evolution of the particle size distribution in the above mentioned
processes. The coagulation-fragmentation equations are integro-partial differential equations which describe the particle size distribution $g(y, t)$ of particles of volume $y>0$ at time $t\geq 0$ and  read as, see \cite{Escobedo:2003, Escobedo:2002, Giri:2011I, Stewart:1989, Stewart:1990},
\begin{align}\label{1cfe}
\hspace{-.2cm}\frac{\partial g(y,t)}{\partial t}  = &\frac{1}{2} \int_{0}^{y} K(y-z,z)g(y-z,t)g(z,t)dz - \int_{0}^{\infty}K(y,z)g(y,t)g(z,t)dz\nonumber\\
  &+\int_0^{\infty}\hspace{-.05cm}F(y, z) g(y+z,t)dz- \frac{1}{2} \int_{0}^{y}\hspace{-.05cm}F(y-z,z) g(y,t)dz:=\rho(g)~~\mbox{(say)},
\end{align}
with the initial datum
\begin{eqnarray}\label{1in1}
g(y,0) = g^{in}(y)\geq 0~ \mbox{a.e.},
\end{eqnarray}
where $\rho(g):=\rho_1(g)-\rho_2(g)+\rho_3(g)-\rho_4(g)$. For $i=1,2,3,4$, $\rho_i(g)$ represent the first, second, third and fourth terms respectively on the right-hand side to (\ref{1cfe}). The coagulation kernel, $K(y,z)$, describes the rate at which particles of volumes $y$ unite with particles of volume $z$ to produce larger particles of volume $(y+z)$.\\

The kernel $F$ represents the rate at which  particles of volume $(y+ z)$ breakup into those of volumes $y$ and $z$. This fragmentation kernel $F$ becomes $F(y-z, z)=0$, if $y<z$. It will be assumed throughout the paper that kernels $K$ and  $F$ are nonnegative measurable unbounded functions with $K(y,z)=K(z,y)$ and $F(y, z) = F(z, y)$ for all $y> 0$ and $z> 0$, i.e., symmetric.\\

The first integral $\rho_1(g)$ on the right-hand side of (\ref{1cfe}) represents the formation of particles of volume $y$ after coalescence of particles of volumes $y-z$ and $z$, whereas the second integral $\rho_2(g)$ shows the disappearance  of particles of volume $y$ after combining with particles of volume $z$. The third and fourth integrals, $\rho_3(g)$ and $\rho_4(g)$, respectively describe the appearance and disappearance of particles of volume $y$ due to fragmentation events.\\

An important property of the solution to CFEs (\ref{1cfe})--(\ref{1in1}) is known as mass conservation i.e. the total mass of particles remains conserved in the system, i.e.
 \begin{eqnarray*}
\int_0^{\infty}y g(y,t)dy=\int_0^{\infty}y g^{in}(y)dy,  \ \ t \geq 0.
\end{eqnarray*}
We know that the total mass of particles is neither created nor destroyed during the coagulation-fragmentation events. Therefore, it is expected that the total volume (mass) remains conserved during these events.  However, when coagulation kernel increases sufficiently rapidly compared to the fragmentation kernel for large volume of particles, a runaway growth takes place to produce an infinite  \emph{gel} (super particle) in finite time which are removed from the system. Therefore, the total mass (volume) of the system breaks down, this phenomenon is known as \emph{gelation}, see \cite{Leyvraz:1983, Leyvraz:1981} and the time at which this process starts is known as \emph{gelation time}.\\

The purpose of this note is to show the existence of mass conserving solution of continuous CFEs (\ref{1cfe})--(\ref{1in1}) by choosing a suitable non-conservative approximation of continuous CFEs. There are several mathematical results available on the existence and uniqueness of solutions to continuous CFEs (\ref{1cfe})--(\ref{1in1}) which have been established by using various techniques under different growth conditions on coagulation and fragmentation kernels, see \cite{Dubovskii:1996, Giri:2013I, Giri:2011, Giri:2012, Stewart:1989, Stewart:1990}. In \cite{Stewart:1989}, Stewart has discussed the existence of the weak solution to CFEs (\ref{1cfe})--(\ref{1in1}) for coagulation kernels satisfying  $K(y,z)\leq k[(1+y)^{\alpha}+(1+z)^{\alpha}]$, where $\alpha \in [0,1[$, for some $k>0$ and fragmentation kernels satisfying $F(y,z)\leq c(1+y+z)^{\beta}$, where  $\beta \in [0, 1[$ for some constant $c>0$. A uniqueness result to continuous CFEs (\ref{1cfe})--(\ref{1in1}) has been established by Stewart \cite{Stewart:1990} for unbounded kernels $K$ and $F$, where kernels $K$ and $F$ satisfy the following conditions
 \begin{eqnarray*}
 \hspace{-1.5cm}K(y,z)\leq k(1+y)^{1/2}(1+z)^{1/2}
 \end{eqnarray*}
 and\begin{eqnarray*}
\int_0^y(1+z)^{1/2}F(y-z,z)dz \leq m(1+y)^{1/2}.
\end{eqnarray*}
 Several authors have also discussed the existence of mass-conserving solutions, when $K(y,z)\leq A(1+y+z)$ for some $A>0$, under various assumptions on fragmentation kernels, see \cite{Dubovskii:1996, Laurencot:2002L}. Particularly, in \cite{Laurencot:2002L}, Lauren\c{c}ot and Mischler have considered fragmentation kernels as $F(y,z)\leq k(1+y+z)$ to the existence of solution to the discrete version of (\ref{1cfe})--(\ref{1in1}). In \cite{Escobedo:2003}, Escobedo et. al. have also shown the existence of mass-conserving solutions to  (\ref{1cfe})--(\ref{1in1}) under strong fragmentation for bilinear growth condition on coagulation kernels. The proof relies on the weak $L^1$ Compactness method  which is originally introduced by Stewart \cite{Stewart:1989}. In \cite{Dubovskii:1996}, Dubovski and Stewart have discussed the existence of mass conserving solutions to CFEs by using a different approach. The classes of coagulation and fragmentation kernels, which they have considered, also cover the hypotheses $(H1)-(H4)$ considered in this article. In addition, they have shown the uniqueness of solution to CFEs under the following additional restriction on  fragmentation kernels
 \begin{eqnarray*}
 \int_0^y F(y-z, z)dz \leq b(1+y)^{m_1}, ~~~~\mbox{for}~~m_1\leq 1.
 \end{eqnarray*}

 Due to the unavailability of a uniqueness result to (\ref{1cfe})--(\ref{1in1}) for $K$ and $F$ satisfying $(H3)$ and  $(H4)$ respectively, it is unclear that the solution to (\ref{1cfe})--(\ref{1in1}) provided by a non-conservative approximation is mass conserving or not.  In general, it is known that a non-conservative approximation to Smoluchowski coagulation equations (SCEs) may lead to the gelation phenomenon.   In 2004,  Filbet and Lauren\c{c}ot \cite{Filbet:2004II} have developed a finite volume scheme to demonstrate the occurrence of gelation to SCEs by using a non-conservative approximation. Moreover, they have seen from some experimental results that for the large computational domain the loss of mass is decreased. Therefore, it is expected that the non-conservative approximation may also give a mass conserving solution to SCEs as the upper limit of the domain of truncation tends to infinity. Further, they have provided a mathematical proof for this observation in \cite{Filbet:2004I}. Later in 2008, Bourgade and Filbet \cite{Bourgade:2008} have generalized the finite volume scheme of \cite{Filbet:2004II} to CFEs. By performing some numerical computations, they have also studied the occurrence of gelation to CFEs which appeared due to the finite interval of computations. In case of CFEs, they have concluded similar observations as in \cite{Filbet:2004II} that the loss of mass can be decreased by taking a sufficiently large computational domain. This gives a positive hope to achieve a mass conserving solution for CFEs as well by considering a non-conservative truncation. Therefore, the aim of this article is to show mathematically that the mass lost due to the non-conservative truncation converges to zero in the limiting case. This result is not very surprising but certainly important from the point of view of mathematical clarity and correctness. The motivation of the present work is from \cite{Ball:1990}, \cite{Bourgade:2008}, \cite{Filbet:2004I} and \cite{Filbet:2004II}.\\

Let us provide a brief plan of the article. In section 2, we mention some hypotheses, definitions, conservative and non-conservative approximations to CFEs. In addition, the existence of mass-conserving solution to CFEs by considering a conservative approximation is recalled. This section also contain the existence of solutions of the non-conservative truncation of CFEs. These solutions may not satisfy the mass conserving property. At the end of section 2, the main result on the existence of mass conserving solution to CFEs with non-conservative approximation is stated. In section 3, the Dunford-Pettis theorem is applied to the family of solutions of non-conservative truncations to CFEs. Further, equicontinuity argument with respect to time  helps us to use a refined version of \textit{Arzel\`{a}-Ascoli theorem}. Moreover, the main existence theorem is proved in this section.

%The title of your section 2
%\section{Examples}
 \section{Preliminaries and main result}
 In order to prove the Theorem \ref{thm1} for the existence of mass conserving solution to (\ref{1cfe})--(\ref{1in1}), we consider the following hypotheses.
\begin{hypp}\label{hyppmcs}
(H1) $K$ and $F$ are non-negative measurable functions on $]0,\infty[ \times ]0,\infty[$,\\
\\
(H2) $K$ and $F$ are symmetric, i.e. $K(y,z)=K(z,y)$  and $F(y,z)=F(z,y)$ for all $(y,z) \in ]0,\infty[ \times ]0,\infty[ $, \\
\\
(H3) $K(y,z)\leq k_1(1+y+z)$ for all $(y,z)\in ]0,\infty[ \times ]0,\infty[ $ and for some constant $k_1>0$,\\
\\
(H4) $ F(y, z)\leq k_2(1+y+z),$ for all  $ (y,z)\in ]0,\infty[ \times ]0,\infty[ $  where  $ k_2 > 0.$
\end{hypp}

\begin{defn}\label{definition}
 A non-negative real valued function $g=g(y,t)$ is a weak solution to (\ref{1cfe})--(\ref{1in1}), if $g \in \mathcal{C}([0, \infty);w- L^1(0, \infty))\bigcap L^{\infty}(0, \infty;L^1_1(0, \infty))$ and $\rho_i(g)\in L^1((0,M)\times (0,T))$, for $M>0$, $T>0$ and $\{i=1,\cdot, \cdot, 4\}$, and
\begin{align}\label{wscfe}
&\int_0^{\infty}[ g(y,t) - g^{in}(y)]\omega(y)dy\nonumber\\=&\frac{1}{2}\int_0^t \int_0^{\infty} \int_{0}^{\infty}\tilde{\omega}(y,z) K(y,z)g(y,s)g(z,s)dzdyds\nonumber\\
&-\frac{1}{2}\int_0^t \int_0^{\infty} \int_0^{\infty} \tilde{\omega}(y,z)  F(y, z) g(y+z,s)dzdyds,
\end{align}
and the last integral on the right-hand side to (\ref{wscfe}) can also be written in the following way:
\begin{eqnarray*}
-\frac{1}{2}\int_0^t \int_0^{\infty} \int_0^{\infty} \tilde{\omega}(y,z)  F(y, z) g(y+z,s)dzdyds=\frac{1}{2}\int_0^t \int_0^{\infty} k_{\omega}(y)g(y,s)dyds,
\end{eqnarray*}
where
\begin{eqnarray}\label{omega1}
\hspace{-.9cm}\tilde{\omega} (y,z):=\omega(y+z)-\omega (y)-\omega(z)
\end{eqnarray}
and
\begin{eqnarray}\label{komega}
 k_{\omega}(y):=-\int_0^y F(z, y-z)\tilde{\omega} (z,y-z) dz,
\end{eqnarray}
for every $t > 0$ and $\omega \in \mathcal{C}_c^{\infty}(0, \infty)$, where $L_1^1(0, \infty):=L^1((0, \infty ) ; (1+y)dy)$ and $\mathcal{C}_c^{\infty}(0, \infty)$ is the space of all infinitely continuously differentiable functions with compact support.
\end{defn}

Now we define the characteristic function $\chi_{\mathbb{E}}$ on a set $\mathbb{E}$ as
 \begin{displaymath}
   \chi_\mathbb{E}({y}) := \left\{
     \begin{array}{lr}
       1 & ~ \mbox{if} ~ {y}\in \mathbb{E},\\
       0 & ~ ~ \mbox{if}~ {y} \notin \mathbb{E}.
     \end{array}
   \right.
\end{displaymath}

 Next, we construct a mass conserving solution relying on the conservative approximation to CFEs (\ref{1cfe})--(\ref{1in1}), which is defined as: for a given  natural number $n\in\mathbb{N}$, we set
\begin{eqnarray*}
g_n^{in}(y):=g^{in}(y)\chi_{]0,n[}(y), ~~~~~~~K_n^c(y,z):=K(y,z)\chi_{]0,n[}(y+z)
\end{eqnarray*}
and
\begin{eqnarray*}
\hspace{-4cm}F_n^c(y, z):=F(y,z)\chi_{]0,n[}(y+z),
\end{eqnarray*}
which gives the following conservative approximation to (\ref{1cfe})--(\ref{1in1}) as:
\begin{align}\label{tcfe}
\frac{\partial \tilde{g}_n(y,t)}{\partial t}  = &\frac{1}{2} \int_{0}^{y} K(y-z,z)\tilde{g}_n(y-z,t)\tilde{g}_n(z,t)dz - \int_{0}^{n-y}K(y,z)\tilde{g}_n(y,t)\tilde{g}_n(z,t)dz\nonumber\\
  &+\int_0^{n-y}F(y, z) \tilde{g}_n(y+z,t)dz- \frac{1}{2} \int_{0}^{y}F(y-z,z) \tilde{g}_n(y,t)dz,
\end{align}
with the truncated initial condition
\begin{eqnarray}\label{1tnin1}
\tilde{g}_n(y,0)=g_n^{in},  ~~ \text{for}~~ y\in ]0,n[.
\end{eqnarray}
Considering $(H1)-(H4)$  and $g_n^{in} \in L^1_1(0, \infty)$, we may show as in \cite{Stewart:1989}, there exists a unique solution $g \in \mathcal{C}( [0, \infty ); L^1(0,n) )$ to (\ref{tcfe})--(\ref{1tnin1}) such that
\begin{eqnarray*}
\int_0^{n} y\tilde{g}_n(y,t)dy= \int_0^{n} y g_n^{in}(y)dy ~~~~ \text{for~all} ~~t\geq 0.
\end{eqnarray*}

Again using $(H1)-(H4)$, there is a subsequence $(\tilde{g}_{n_k})$ of $(\tilde{g}_n)$ such that
 \begin{eqnarray*}
\tilde{g}_{n_k} \rightarrow g, ~~~\mbox{in}~~\mathcal{C}([0,T]; w-L^1_1(0, \infty))~~~\mbox{as}~~n_k \to \infty.
 \end{eqnarray*}
 Hence, $g$ is indeed a solution to (\ref{1cfe})--(\ref{1in1}). Consequently, $g$ is mass conserving, i.e
\begin{eqnarray*}
M_1(t):=\int_0^{\infty}yg(y,t)dy=\int_0^{\infty}yg^{in}(y)dy:=M_1^{in}.
\end{eqnarray*}

 Here, the space of all weakly continuous functions from $[0, T]$ to $L_1^1(0, \infty )$ is denoted by $ \mathcal{C}( [0,T]; w-L^1_1(0, \infty))$  and  if
\begin{eqnarray*}
\lim_{n \to \infty} \sup_{t \in [0,T[}\bigg|\int_0^{\infty}(1+y)[g_n(y,t)-g(y,t)] \omega(y)dy\bigg|=0,
\end{eqnarray*}
for every $\omega \in L^{\infty}(0, \infty )$, then we say that a sequence $(g_n)$ converges to $g$ in $\mathcal{C}( [0,T]; w-L^1_1(0, \infty ))$.\\

There is a possibility to have different approximations to CFEs (\ref{1cfe})--(\ref{1in1}) which are not the conservative one i.e. (\ref{tcfe}), see \cite{Bourgade:2008}. However, in order to study the gelation phenomenon, a non-conservative approximation of coagulation and a conservative truncation of fragmentation is required which can be constructed as follows: for a given $n\in \mathbb{N}$, we define
\begin{eqnarray*}
g_n^{in}(y):= g^{in}(y)\chi_{]0,n[}(y),~ K_n^{nc}(y,z):=K(y,z)\chi_{]0,n[}(y)\chi_{]0,n[}(z)
\end{eqnarray*}
 and
\begin{eqnarray*}
\hspace{-4.8cm} F_n^c(y,z):=F(y,z)\chi_{]0,n[}(y+z),
\end{eqnarray*}
which gives the following non-conservative approximation to (\ref{1cfe})--(\ref{1in1})
\begin{align}\label{tncfe}
\frac{\partial g_n(y,t)} {\partial t}  = &\frac{1}{2} \int_{0}^{y} K(y-z,z)g_n (y-z,t)g_n (z,t)dz - \int_{0}^{n}K(y,z)g_n(y,t)g_n(z,t)dz\nonumber\\
  &+\int_0^{n-y}F(y, z) g_n(y+z,t)dz- \frac{1}{2} \int_{0}^{y}F(y-z,z) g_n(y,t)dz\nonumber\\
   :=&\rho^n{(g_n)}~\mbox{(say)},
\end{align}
with the truncated initial condition
\begin{eqnarray}\label{1nin1}
g_n(y,0)=g_n^{in},  ~~ \text{for} ~~y\in ]0,n[,
\end{eqnarray}
where $\rho^n:=\rho_1^n-\rho_2^n+\rho_3^n-\rho_4^n$, and $\rho_1^n$, $\rho_2^n$, $\rho_3^n$ and $\rho_4^n$ represent the first, second, third and fourth integrals respectively on the right-hand side to (\ref{tncfe}).\\

Now, it can easily be verified from (\ref{tncfe})--(\ref{1nin1}) that the total mass may not remain conserved. i.e.
\begin{align}\label{tncfe1}
\int_0^n{ yg_n(y,t)}dy =&\int_0^n{ yg_n^{in}(y)}dy\nonumber\\
&-\frac{1}{2}\int_0^t \int_0^n \int_{n-y}^n (y+z)K(y,z)g_n(y,s)g_n(z,s)dzdyds.
\end{align}
This implies that
\begin{eqnarray}\label{tncfe2}
\int_0^n{ yg_n(y,t)}dy \leq \int_0^n{ yg_n^{in}(y)}dy.
\end{eqnarray}
 From (\ref{tncfe1}), we interpret that the total volume (mass) may not be conserved.\\ Similarly, we show the existence and uniqueness of a non-negative solution $g_n \in \mathcal{C}([0,\infty); L^1(0,n))$ to (\ref{tncfe})--(\ref{1nin1}) by using a classical fixed point theorem, but this solution $g_n$ does not satisfy mass conserving property.\\

For the sake of information, we would like to mention that there is also another non-conservative approximation which is different from above. In this approximation, non-conservative form of coagulation and non-conservative form of fragmentation are considered, see \cite{Bourgade:2008}. The non-conservative coagulation and non-conservative fragmentation equation is given by
\begin{align*}
\frac{\partial \tilde{g_n}(y,t)}{\partial t}  = &\frac{1}{2} \int_{0}^{y} K(y-z,z)\tilde{g_n}(y-z,t)\tilde{g_n}(z,t)dz - \int_{0}^{n}K(y,z)\tilde{g_n}(y,t)\tilde{g_n}(z,t)dz\nonumber\\
  &+\int_0^{n}F(y, z) \tilde{g_n}(y+z,t)dz- \frac{1}{2} \int_{0}^{y}F(y-z,z) \tilde{g_n}(y,t)dz.
\end{align*}
In the non-conservative fragmentation part, one can see that the mass of the system with respect to time increases which is not realistic in nature but in case of non-conservative coagulation the total mass decreases in time. Now, the sign of the  addition between these two increasing and deceasing masses is very difficult to determine. So, here we have considered the non-conservative coagulation and conservative fragmentation.\\

Now, we are in a position to state the main theorem of this paper.
\begin{thm}{(Main Theorem)}\label{thm1}
 Assume that hypotheses $(H1)-(H4)$ hold with  the initial datum $g^{in}\in L_1^1{(0, \infty)}$. For $n\geq 1$, we denote $g_n$ the solution to (\ref{tncfe})--(\ref{1nin1}). Then there is a subsequence $(g_{n_k})$ of $(g_{n})$ and a solution $g$ to (\ref{1cfe})--(\ref{1in1}) such that
 \begin{eqnarray*}
 g_{n_k}\to g  ~~~~~~~~  \text{in}~~\mathcal{C}( [0,T]; w-L^1_1(0, \infty )) ~~\text{for ~each}~ T>0
 \end{eqnarray*}
 satisfying the formulation (\ref{wscfe}). Moreover, it satisfies the mass conserving property, i.e.
 \begin{eqnarray*}
 \int_0^{\infty}yg(y,t)dy=\int_0^{\infty}yg^{in}(y)dy.
 \end{eqnarray*}
 \end{thm}

  In order to prove the Theorem \ref{thm1}, let $g^{in}\in L_1^1(0, \infty )$, from a refined version of de la Vall\'{e}e-Poussin theorem (see \cite{Dellacherie:1975, Filbet:2004I}), we are cognizant that there exist two non-negative convex functions $\sigma_1$ and $\sigma_2$ in $\mathcal{C}^2[0,\infty)$ (space of all twice continuously differentiable functions) such that their derivatives,  $\sigma_1^{'}$ and $\sigma_2^{'}$ are concave with
\begin{eqnarray}\label{convex1}
\sigma_i(0)=0,~~~\lim_{r \to {\infty}}\frac{\sigma_i(r)}{r}=\infty,~~~~i=1,2
\end{eqnarray}
and
\begin{eqnarray}\label{convex2}
\int_0^{\infty}\sigma_1(y)g^{in}(y)dy<\infty,~~~\text{and}~~~\int_0^{\infty}{\sigma_2(g^{in}(y))}dy<\infty.
\end{eqnarray}
Let us state some properties of non-decreasing convex function with concave derivatives, which are required to prove our main result.
\begin{lem}\label{convexlemma}
 Let $r_1, r_2 \in (0, \infty )$ such that, we have the following results
\begin{eqnarray}\label{convex3}
\hspace{-4.7cm}\sigma_2(r_1)\leq r_1\sigma^{'}_2(r_1)\leq 2\sigma_2(r_1)
\end{eqnarray}
and
\begin{eqnarray}\label{convex4}
0\leq \sigma_2(r_1+r_2)-\sigma_2(r_1)-\sigma_2(r_2)\leq 2\frac{r_1\sigma_2(r_2)+r_2\sigma_2(r_1)}{r_1+r_2}.
\end{eqnarray}
\end{lem}
\begin{proof}
The proof of Lemma \ref{convexlemma} is straight forward, see \cite{Laurencot:2001} Lemma A.1, and Lemma A.2.
\end{proof}
We next define a weak formulation for the non-conservative approximation (\ref{tncfe})--(\ref{1nin1}) of CFEs,  for $n\geq 1$, and $\omega \in L^{\infty}(0, \infty )$ as
\begin{align}\label{nctp3}
\int_0^n[g_n(y,t)-g_n^{in}(y)]\omega(y)dy=& \frac{1}{2} \int_0^t \int_0^n\int_0^{n} G_{\omega}(y,z)K(y,z) g_n(y,s)g_n(z,s)dzdyds\nonumber\\
& +\frac{1}{2}\int_0^t\int_0^n k_{\omega}(y)g_n(y,s)dyds,
\end{align}
or
 \begin{align}\label{nct3}
&\int_0^{n}[ g_n(y,t) - g^{in}_n(y)]\omega(y)dy\nonumber\\=&\frac{1}{2}\int_0^t \int_0^{n} \int_{0}^{n} G_{\omega}(y,z) K(y,z)g_n(y,s)g_n(z,s)dzdyds\nonumber\\
&-\frac{1}{2}\int_0^t \int_0^{n} \int_0^{n-z} \tilde{\omega}(y,z)  F(y, z) g_n(y+z,s)dydzds,
\end{align}
where $$ G_{\omega}(y,z):= \omega (y+z)\chi_{]0,n[}(y+z)-\omega(y)-\omega(z), $$
$\tilde{\omega}(y,z)$ and $k_{\omega}(y) $ are defined in (\ref{omega1}) and (\ref{komega}) respectively.

\section{Weak Compactness}
%%%%%%%%%%%%%%%%%%%%%%%%%%%%%%%%%%%%%%%%%%%(LEMMA)%%%%%%%%%%%%%%%%%%%%%%%%%%%%%%%
\begin{lem}\label{V(T)}
Let $(H1)-(H4)$ hold and $g^{in}_n \in L_1^1(0,\infty )$. Suppose $g_n$ satisfies (\ref{nctp3}). Then for $T>0$,  there is a  constant $V(T)$ depending on $T$ such that the following inequality holds
\begin{eqnarray*}
\int_0^n(1+y) g_n(y,t)dy \leq V(T).
\end{eqnarray*}
\end{lem}
\begin{proof}
Let us simplify the following integral by using (\ref{tncfe2}) as
\begin{align}\label{bound0}
\int_0^n (1+y)g_n(y,t)dy=&\int_0^1 g_n(y,t)dy+\int_1^n g_n(y,t)dy+\int_0^n yg_n(y,t)dy\nonumber\\
 \leq & \int_0^1 g_n(y,t)dy+2\int_0^n yg_n^{in}(y)dy.
\end{align}
Set $\omega(y):=\chi_{]0,1[}(y)$ and substituting it into (\ref{nctp3}) to have
\begin{align}\label{bound1}
\int_0^1 [g_n(y,t)-g_n^{in}(y)] dy =&\frac{1}{2} \int_0^t \int_0^n\int_0^{n} G_{\omega}(y,z)K(y,z)g_n(y,s)g_n(z,s)dzdyds\nonumber\\
& +\frac{1}{2}\int_0^t\int_0^n k_{\omega}(y)g_n(y,s)dyds.
\end{align}
Now, the first integral on the right-hand side of (\ref{bound1}) can be split into following six sub-integrals
\begin{align}\label{Equibound4}
&\frac{1}{2}\int_0^t \int_0^n\int_0^n G_{\omega}(y,z)K(y,z)g_n(y,s)g_n(z,s)dzdyds\nonumber\\
=& \frac{1}{2}\int_0^t \biggl\{ \int_0^1\int_0^{1-y}+\int_0^1\int_{1-y}^1+\int_0^1\int_1^{n-y}+ \int_1^n\int_0^{1}+ \int_1^n\int_1^{n-y}\nonumber\\
&\hspace{1cm}+\int_0^n\int_{n-y}^n   \biggr\}G_{\omega}(y,z)K(y,z)g_n(y,s)g_n(z,s)dzdyds.
\end{align}
From the definition of $G_{\omega}(y,z)$ and $\omega(y)=\chi_{]0,1[}(y)$, one can easily show that all the above sub-integrals on the right-hand side to (\ref{Equibound4}) are less than or equal to zero, which ensures the non-positivity of the first integral on the right-hand side to (\ref{bound1}). Next, the second integral of (\ref{bound1}) can be written as
\begin{align}\label{Equibound21}
\frac{1}{2}\int_0^t\int_0^n k_{\omega}(y)g_n(y,s)dyds=& \frac{1}{2}\int_0^t\int_0^1 k_{\omega}(y)g_n(y,s)dyds\nonumber\\
&+\frac{1}{2}\int_0^t\int_1^n k_{\omega}(y)g_n(y,s)dyds.
\end{align}
Now, using $(H4)$, (\ref{tncfe2}) and $\omega(y)=\chi_{]0,1[}(y)$, the following sub-integral of (\ref{Equibound21}) can be estimated as
\begin{align}\label{bound14}
&\frac{1}{2}\int_0^t\int_0^1 k_{\omega}(y) g_n(y,s)dyds \nonumber\\
=& -\frac{1}{2}\int_0^t\int_0^1  \int_0^y F(z, y-z)[\chi_{]0,1[}(y)-\chi_{]0,1[}(y-z)-\chi_{]0,1[}(z)]g_n(y,s)dz dyds\nonumber\\
 \leq &\frac{1}{2} k_2 \int_0^t\int_0^1 \int_0^y (1+y)g_n(y,s)dzdyds
\leq k_2T \int_0^{n} y g_n^{in}(y)dy.
\end{align}

Further, by taking $\omega(y)=\chi_{]0,1[}(y)$, the last sub-integral of (\ref{Equibound21}) is calculated as
\begin{align}\label{bound15}
&\frac{1}{2}\int_0^t\int_1^n k_{\omega}(y)g_n(y,s) dyds\nonumber\\
=&\frac{1}{2}\int_0^t\int_1^n  \int_0^y F(z, y-z)\chi_{]0,1[}(y-z)g_n(y,s)dz dyds\nonumber\\
&+\frac{1}{2}\int_0^t\int_1^n  \int_0^y F(z, y-z)\chi_{]0,1[}(z)g_n(y,s)dz dyds.
\end{align}
Using the transformation $y-z=z^{'}$, $y=y^{'}$ and the symmetry of $F$ to the first integral on the right-hand side to (\ref{bound15}), this can easily be observed that both integrals on the right-hand side are equal. Therefore, (\ref{bound15}) becomes
\begin{eqnarray*}
\frac{1}{2}\int_0^t\int_1^n k_{\omega}(y)g_n(y,s)dyds=\int_0^t\int_1^n  \int_0^1 F(z, y-z)g_n(y,s)dzdyds,
\end{eqnarray*}
which can further be estimated, by using $(H4)$, as
\begin{align}\label{bound16}
\hspace{-.2cm}\frac{1}{2}\int_0^t\int_1^n k_{\omega}(y)g_n(y,s)dyds \leq & 2k_2\int_0^t\int_1^n y g_n(y,s)dyds \leq  2k_2T\int_0^n y g_n^{in}(y)dy.
\end{align}

Calculating (\ref{Equibound21}) from (\ref{bound14}) and (\ref{bound16}), and using the non-positivity of the integral on the left-hand side of (\ref{Equibound4}), (\ref{bound1}) can be further simplified as
\begin{eqnarray}\label{bound17}
\int_0^1 g_n(y,t)dy \leq \int_0^1 g_n^{in}(y)dy+3k_2T\int_0^n y g_n^{in}(y)dy.
\end{eqnarray}
Now, using  (\ref{bound17}) into (\ref{bound0}) and  $g^{in}_n \in L_1^1(0,\infty )$, we thus have
\begin{eqnarray*}
\int_0^n(1+y) g_n(y,t)dy \leq \int_0^1 g_n^{in}(y)dy +(2+3k_2T)\int_0^n y g_n^{in}(y)dy \leq V(T),
\end{eqnarray*}
where $V(T):= \int_0^1 g_n^{in}(y)dy+(2+3k_2T)M_1^{in}$. This completes the proof of Lemma \ref{V(T)}.
\end{proof}
In order to prove the next lemma, we require one more important property of convex functions, i.e. the non-decreasing convex function $\sigma_2 $ satisfies for $r, s \in (0, \infty )$
\begin{eqnarray}\label{M(T)1}
r\sigma_2{^{'}}(s)\leq \sigma_2(r)+\sigma_2(s).
\end{eqnarray}
Note that (\ref{M(T)1}) follows from  (\ref{convex3}) and  the convexity of $\sigma_2$.

In the following lemma, the equi-integrability of $\{g_n\}_{n\geq 1} \subset L^1_1(0,R)$ is shown by using the convex function $\sigma_2$.
\begin{lem}\label{last}
For any $T>0$ and $R \in (0, n)$, there are two constants $C(R, T)$ and $C_4(R,T)$ such that
\begin{eqnarray*}
\hspace{-2.7cm}(i)  \sup_{t\in [0,T]}\int_0^R \sigma_2(g_n(y,t))dy \leq C(R,T),
\end{eqnarray*}
\begin{eqnarray*}
(ii)  \sup_{t\in [0,T]}\bigg| \frac{d}{dt}\int_0^R g_n(y,t)\omega(y)dy \bigg|\leq C_4(R,T)\|\omega\|_{L^{\infty}(0,R)},
\end{eqnarray*}
for every $n\geq 1$ and $\omega \in L^{\infty}(0,R),$ where $\sigma_2$ is the convex function with concave derivative satisfying (\ref{convex1}) and (\ref{convex2}).
\end{lem}
\begin{proof} For any $R <n$, using (\ref{1cfe})--(\ref{1in1}), Leibniz's rule and Fubini's theorem, we have
\begin{align}\label{eqint2}
\frac{d}{dt}\int_0^R \sigma_2(g_n(y,t))dy \leq &\frac{1}{2} \int_0^R \int_0^{R-z} \sigma_2^{'}(g_n(y+z,t))K_n(y,z)g_n(y,t)g_n(z,t)dydz\nonumber\\
&+\int_0^n \int_0^{R} \sigma_2^{'}(g_n(y,t)) F_n(y, z-y) g_n(z,t)dydz.
\end{align}
Now, we estimate each term on the right-hand side of (\ref{eqint2}) individually. Calculating the first term on the right-hand side of (\ref{eqint2}), by using $(H3)$, (\ref{M(T)1}) and Lemma \ref{V(T)}, as follows
\begin{align}\label{est1}
&\frac{1}{2} \int_0^R \int_0^{R-z}  \sigma_2^{'}(g_n(y+z,t))K_n(y,z)g_n(y,t)g_n(z,t)dydz\nonumber\\
\leq & \frac{1}{2}k_1 \int_0^R \int_0^{R-z}(1+y+z)[\sigma_2(g_n(y+z,t))+\sigma_2(g_n(z,t))]g_n(y,t)dydz\nonumber\\
\leq & k_1 (1+R)\int_0^R \int_0^{R}\sigma_2(g_n(z,t))g_n(y,t)dydz\nonumber\\
\leq & C_1(R, T)\int_0^R\sigma_2(g_n(y,t))dy,
\end{align}
where $C_1(R, T):=k_1(1+R)V(T)$. By using (\ref{M(T)1}), $(H4)$ and Lemma \ref{V(T)}, the second term on the right-hand side to (\ref{eqint2}) can be estimated as
\begin{align}\label{est2}
&\int_0^n \int_0^{R}  \sigma_2^{'}(g_n(y,t)) F_n(y, z-y) g_n(z,t)dydz\nonumber\\
\leq &k_2 V(T) \int_0^{R} [\sigma_2(g_n(y,t))+\sigma_2(1)]dy\nonumber\\
\leq & C_2(T)\int_0^{R} \sigma_2(g_n(y,t))dy+ C_3(R, T),
\end{align}
where $C_2(T):=k_2V(T)$ and $C_3(R, T):=k_2V(T)R\sigma_2(1)$. Inserting (\ref{est1}) and (\ref{est2}) into (\ref{eqint2}), we find
\begin{align*}
\frac{d}{dt}\int_0^R \sigma_2(g_n(y,t))dy\leq &C_1(R, T)\int_0^R\sigma_2(g_n(y,t))dy\nonumber\\
&+C_2(T)\int_0^R\sigma_2(g_n(y,t))dy+C_3(R, T).
\end{align*}
Then applying Gronwall's inequality, monotonicity of $\sigma_2$ and $g_n^{in}\leq g^{in}$, we obtain
\begin{eqnarray*}
\int_0^R \sigma_2(g_n(y,t))dy\leq C(R, T),
\end{eqnarray*}
where $C(R, T)$ is a constant depending on $R$ and $T$.
This completes the proof of Lemma \ref{last} $(i)$. In order to prove the second part of Lemma \ref{last}, the following term is estimated, by using Leibniz's rule, $(H3)$, $(H4)$ and Fubini's theorem, as
\begin{align*}
\bigg|\frac{d}{dt}\int_0^R  g_n(y,t)\omega(y)dy \bigg | \leq  &\frac{1}{2}  \int_0^R\int_0^{R-y} |G_{\omega} (y, z)|  K_n(y,z)g_n(y,t)g_n(z,t)dydz\nonumber\\
&+\int_0^R \int_{R-y}^n |{\omega}(y)|  K_n(y,z)g_n(y,t)g_n(z,t)dzdy\nonumber\\
&+\frac{1}{2}\int_0^R\int_0^y |\omega(y)| F_n(y-z,z)g_n(y,t)dzdy\nonumber\\
&+\int_0^R\int_0^{n-y} |\omega(y)| F_n(y, z) g_n(y+z,t)dzdy\nonumber\\
\leq & C_4(R, T)\|\omega\|_{L^{\infty}]0,R[},
\end{align*}
where $C_4(R, T):=5/2k_1(1+R)V(T)^2+3/2k_2RV(T)$. This completes the proof of Lemma \ref{last} $(ii)$.
\end{proof}
Next, the following result is established to study the behaviour of $g_n$ for large values of $y$.
\begin{lem}\label{massest1}
Let $(H1)-(H4)$ hold and the initial data $g_n^{in}(y) \in L_1^1(0,\infty )$. Then, for $T>0$, there is a constant $C(T)$ depending on $T$ such that,
\begin{eqnarray*}
\hspace{-3.6cm}(i)~~ \sup_{t\in [0,T]}\int_0^n \sigma_1 (y)g_n(y,t)dy\leq C(T),~~~ \mbox{for~ every}~~n\geq 1,
\end{eqnarray*}
\begin{eqnarray*}
\hspace{-2.7cm}(ii)~~ \int_0^T \int_0^n\int_{n-y}^n \sigma_1 (y)K(y,z)g_n(y,s)g_n(z,s)dzdyds\leq C( T),
\end{eqnarray*}
and
\begin{eqnarray*}
(iii)~~ \int_0^T \int_0^n\int_{0}^{n-y} [\sigma_1 (y+z)-\sigma_1 (y)-\sigma_1 (z)]F(y,z)g_n(y+z,s)dzdyds\leq C( T),
\end{eqnarray*}
where $\sigma_1 $ is the convex function with concave derivative and satisfying (\ref{convex1})--(\ref{convex2}).
\end{lem}

\begin{proof} Setting $\omega (y)=\sigma_1(y)$, for all $y \in (0,n)$ into (\ref{nct3}), we obtain
\begin{align}\label{equality}
&\int_0^n \sigma_1(y)g_n(y,t)dy\nonumber\\
\leq &C+\frac{1}{2}\int_0^t [P_n(s)+Q_n(s)]ds\nonumber\\
&-\frac{1}{2}\int_0^t \int_0^{n} \int_0^{n-z} [\sigma_1(y+z)-\sigma_1(y)-\sigma_1(z)]  F(y, z) g_n(y+z,s)dydzds,
\end{align}
where, from (\ref{convex2}), the monotonicity of $\sigma_1$ and $g_n^{in}\leq g^{in}$, we set
\begin{eqnarray*}
\hspace{-2.9cm}C:=\int_0^{\infty} \sigma_1(y)g^{in}(y)dy<\infty,
\end{eqnarray*}
\begin{eqnarray*}
P_n(s):= \int_0^n \int_0^{n-y}G_{\sigma_1}(y,z)K(y,z)g_n(y,s)g_n(z,s)dzdy
\end{eqnarray*}
and
\begin{eqnarray*}
Q_n(s):= \int_0^n \int_{n-y}^n G_{\sigma_1}(y,z)K(y,z)g_n(y,s)g_n(z,s)dzdy.
\end{eqnarray*}
Using the properties of $\sigma_1$ from (\ref{convex4}) and  hypothesis $(H3)$, we estimate the following term as
\begin{align}\label{est}
K(y,z)G_{\sigma_1}(y,z)= &K(y,z)[\sigma_1 (y+z)\chi_{]0,n[}(y+z)-\sigma_1(y)-\sigma_1(z)]\nonumber\\
 \leq &K(y,z)[ \sigma_1 (y+z)-\sigma_1(y)-\sigma_1(z)]\nonumber\\
 \leq &2k_1(1+y+z) \frac{y\sigma_1(z)+z\sigma_1(y)}{y+z}.
\end{align}
Using (\ref{est}), we estimate $P_n(s)$ defined in (\ref{equality}) as
\begin{align}\label{est3}
\hspace{-.18cm}P_n(s) \leq & 2k_1\int_0^n \int_0^{\min\{1-y, 0\}}(1+y+z) \frac{y\sigma_1(z)+z\sigma_1(y)}{y+z} g_n(y,s)g_n(z,s)dzdy\nonumber\\
&+2k_1\int_0^n \int_{\min\{1-y, 0\}}^{n-y}(1+y+z) \frac{y\sigma_1(z)+z\sigma_1(y)}{y+z} g_n(y,s)g_n(z,s)dzdy.
\end{align}
Further, each integral on the right-hand side of (\ref{est3}) is estimated separately. Let us simplify the first term on the right-hand side to (\ref{est3}) by using  Lemma \ref{V(T)} as
\begin{align}\label{est4}
2k_1\int_0^n \int_0^{\min\{1-y, 0\}}&(1+y+z) \frac{y\sigma_1(z)+z\sigma_1(y)}{y+z} g_n(y,s)g_n(z,s)dzdy\nonumber\\
&\leq  8k_1 V(T) \int_0^n  \sigma_1(y)g_n(y,s)dy.
\end{align}
Again using  Lemma \ref{V(T)} the second integral on the right-hand side of (\ref{est3}) can be estimated as
\begin{align}\label{est5}
2k_1\int_0^n \int_{\min\{1-y, 0\}}^{n-y}&(1+y+z) \frac{y\sigma_1(z)+z\sigma_1(y)}{y+z} g_n(y,s)g_n(z,s)dzdy\nonumber\\
\leq & 4 k_1\int_0^n \int_{\min\{1-y, 0\}}^{n-y}(y+z) \frac{y\sigma_1(z)+z\sigma_1(y)}{y+z} g_n(y,s)g_n(z,s)dzdy\nonumber\\
\leq & 8k_1V(T)\int_{0}^{n} \sigma_1(y) g_n(y,s)dy.
\end{align}

Finally, inserting estimates (\ref{est4}), and (\ref{est5})   into (\ref{est3}), we have
\begin{eqnarray}\label{Pn}
P_n(s) \leq 16k_1 V(T) \int_0^n  \sigma_1(y)g_n(y,s)dy.
\end{eqnarray}
Now the term $Q_n(s)$ defined in (\ref{equality}) is considered. Since $y+z \geq n$ in this term, therefore, we have
\begin{eqnarray}\label{bountct}
G_{\sigma_1}(y,z)=-\sigma_1(y)-\sigma_1(z).
\end{eqnarray}
Using (\ref{bountct}), it is clear that $Q_n(s) \leq 0$.
Moreover, the monotonicity of $\sigma_1^{'}$ and $\sigma_1(0)=0$ which ensure that $ \sigma_1(y+z)\geq \sigma_1(y)+\sigma_1(z)$ and thus we have
\begin{eqnarray}\label{fp}
 -\frac{1}{2}\int_0^t \int_0^{n} \int_0^{n-z}  [\sigma_1 (y+z)-\sigma_1 (y)-\sigma_1 (z)]  F(y, z) g_n(y+z,s)dydzds \leq 0.
\end{eqnarray}
Finally, substituting (\ref{Pn}), $Q_n(s) \leq 0$ and (\ref{fp}) into (\ref{equality}), we estimate
\begin{eqnarray*}
\int_0^n \sigma_1(y)g_n(y,t)dy \leq  C+16k_1V(T)\int_0^t \int_0^n  \sigma_1(y)g_n(y,s)dyds.
\end{eqnarray*}
Therefore, Gronwall's inequality gives $(i)$, i.e.
\begin{eqnarray*}
\int_0^n \sigma_1(y)g_n(y,t)dy \leq C(T),
\end{eqnarray*}
where $C(T)$ depends on $T$. Moreover, $(ii)$ and $(iii)$ in Lemma \ref{massest1} clearly follow from the substitution of (\ref{Pn}), $Q_n(s) \leq 0$ and (\ref{fp})  into (\ref{equality}) and $(i)$.
\end{proof}

Further, for a fixed $T>0$, the equicontinuity of the family $\{g_n(t), t\in[0,T]\}$ with respect to time $t$ in $L^1(0,\infty)$ can be shown similar to \cite{Giri:2011, Giri:2012}.\\

Then from a refined version of \textit{Arzel\`{a}-Ascoli theorem}, see Theorem $2.1$ in  \cite{Stewart:1989} or page $228$ in \cite{Ash:1972}, Lemma \ref{V(T)} and Lemma \ref{last} $(i)$, we conclude that there exist a subsequence (${g_{n_k}}$) and a non-negative function $g\in L^\infty([0,T];L^1(0,\infty))$ such that
\begin{eqnarray*}
\lim_{n_k\to\infty} \sup_{t\in [0,T]}{\left\{ \left| \int_0^\infty  [ g_{n_k}(y,t) - g(y,t)]\ \phi(y)\ dy \right| \right\}} = 0, \label{vittel}
\end{eqnarray*}
for all $T>0$ and $\phi \in L^\infty(0,\infty)$.\\
This implies
\begin{eqnarray}\label{econtresult}
 g_{n_k}(y,t) \rightharpoonup  g(y,t)~~~ \mbox{in} ~~~L^1(0, \infty )  ~~~\mbox{as}~~n_k \rightarrow \infty,
\end{eqnarray}
converges uniformly for all $t \in [0,T]$ to some $g \in \mathcal{C}([0,T]; w-L^1(0, \infty ))$.

\begin{proof}\emph{of Theorem \ref{thm1}}. From de la Vall\'{e}e-Poussin theorem, Lemmas \ref{V(T)}--\ref{massest1}, (\ref{convex1}), Dunford-Pettis theorem and (\ref{econtresult}), we conclude that $(g_n)$ is relatively compact in $\mathcal{C}([0, T]; w-L_1^1(0, \infty ))$ for each $T>0$. There is thus a subsequence of $(g_n)$ (say $g_n$) and a nonnegative function $g\in \mathcal{C}([0, \infty); w-L_1^1(0, \infty ))$ such that
\begin{eqnarray}\label{weakcon1}
g_n \to g ~~~\mbox{in} ~~\mathcal{C}([0,T]; w-L_1^1(0, \infty ))
\end{eqnarray}
for each $T>0$. Then by using $(H3)$, $(H4)$, Definition \ref{definition} and (\ref{weakcon1}), it can easily be shown that each integral on the right-hand of (\ref{tncfe}) converges weakly to each integral on the right-hand of (\ref{1cfe}) respectively, see \cite{Banasiak:2015, Escobedo:2003, Laurencot:2002L, Stewart:1989}, i.e.
\begin{eqnarray*}
\rho_i^n(g_{n}) \rightharpoonup \rho_i(g)~~~~\mbox{weakly~~in}~~L^1((0, M)\times(0,T)),
\end{eqnarray*}
where $i=1,2,3,4$, for each $M>0$ and $T>0$. Thus, this ensures that $g$ is indeed a solution to (\ref{1cfe})--(\ref{1in1}).\\

In order to complete the proof of Theorem \ref{thm1}, we need to show the mass conserving property of the solution to (\ref{1cfe})--(\ref{1in1}).
From (\ref{tncfe1}), it is sufficient to show that
\begin{eqnarray}\label{lastmass}
\lim_{n \to \infty} \frac{1}{2} \int_0^t \int_0^n\int_{n-y}^n (y+z)K(y, z) g_n(y,s) g_n(z,s)dzdyds =0.
\end{eqnarray}
 To prove this result, we use the following argument: if $y+z>n$, then either $y>n/2$ or $z>n/2$ so that $\sigma_1(y)+\sigma_1(z)\geq \sigma_1(n/2)$ for a non-negative and the increasing convex function
$\sigma_1 $.  Finally, (\ref{lastmass}) can easily be shown by using Lemma \ref{massest1} $(ii)$, (\ref{convex1}) and  the same procedure as in \cite{Filbet:2004I}.
This completes the proof of the Theorem \ref{thm1}.
\end{proof}

\section*{Acknowledgments} This work was supported by Faculty Initiation Grant $MTD/$ $FIG/100680$, Indian Institute of Technology Roorkee, Roorkee-247667, India. The authors would also like to thank University Grant Commission (UGC), No. $6405/11/$ $44$, India, for providing Ph.D fellowship to Prasanta Kumar Barik. We also wish to thank the referee for his valuable comments and suggestions that helped to improve the manuscript.

\medskip
% The data information below will be filled by AIMS editorial staff
Received xxxx 20xx; revised xxxx 20xx.
\medskip

\end{document}